\documentclass[12pt]{amsart}
\usepackage{amsmath,amsfonts,amsthm,amssymb,amscd}
\usepackage{color}
\usepackage{hyperref}
\usepackage{tikz-cd}
\usepackage{graphicx}
\usepackage{enumerate}
\usepackage{changepage}

\theoremstyle{definition}
\newtheorem{thm}{Theorem}
\numberwithin{thm}{section}

\newtheorem{prop}[thm]{Proposition}

\newtheorem{lem}[thm]{Lemma}

\newcommand{\abs}[1]{| #1 |}
\newcommand{\trm}[1]{\textrm{#1}}
\newcommand{\al}{\alpha}
\newcommand{\tit}[1]{\textit{#1}}
\def\F{\mathbb{F}}

\newcommand{\tbf}[1]{\textbf{#1}}
\newcommand{\lam}{\lambda}
\newcommand{\ev}{\textrm{ev}}
\newcommand{\non}{\noindent}
\newcommand{\diag}{\textrm{diag}}

\begin{document}

\title[Matrix Waring Problem]{Matrix Waring problem}

\email{kishorekrishna@iittp.ac.in}

\address{Krishna Kishore,
Department of Mathematics,
Indian Institute of Technology Tirupati
Tirupati, AP
India 517506}
\author{Krishna Kishore}

\begin{abstract}
We prove that for all integers $k \geq 1$, there exists a constant $C_k$ depending only on $k$, such that for all $q > C_k$,  and for $n = 1, 2$ every matrix in $M_n(\F_q)$ is a sum of two $k$th powers and for all $n \geq 3$ every matrix in $M_n(\F_q)$ is a sum of \textit{at most three} $k$th powers.
\end{abstract}


\maketitle

\section{Introduction}
Let $n$ be a positive integer, $R$ a commutative ring with unity, and $M_n(R)$ the ring of $n \times n$ matrices over $R$. The Matrix Waring Problem over $R$ is, essentially, to represent an arbitrary matrix in $M_n(R)$ as a sum of \textit{two} $k$th powers for any positive integer $k \geq 1$. This goal of this article is to answer this question in the case where $R$ is a finite field $\F_q$, with $q$ sufficiently large; see Theorem \ref{intro_main_thm_1} and Theorem \ref{intro_main_thm_2}. 

The classical Waring problem states that every natural number is a sum of $4$ squares, of $9$ cubes, of $19$ biquadrates, etc. In other words,  for all integers $k \geq 1$, there exists a smallest positive integer $g(k)$ such that every natural number is a sum of at most $g(k)$ number of $k$th powers (of natural numbers); in particular, $g(2) = 4$, $g(3) = 9$, $g(4) = 19$. The existence of $g(k)$ for arbitrary $k$ was proved by Hilbert, and later on explicit formulae for $g(k)$ have been found for all but finitey many values of $k$; for a survey of results about Waring's problem and its variants see \cite{VW}. In this context, a variant of the classical Waring problem may be posed, namely, for all $k \geq 2$, one may ask whether there exists a smallest bound $G(k)$ such that every \tit{sufficiently large} natural number is a sum of at most $G(k)$ number of $k$th powers. Unlike the case of $g(k)$, the values of $G(k)$ have been found only when $k = 2$ and $k =4$.

Waring's problem makes sense over arbitrary rings, and in particular for matrix rings over finite fields. Let $k$ and $n$ be positive integers and let $G(k,n)$ be the smallest positive integer such that \tit{for all sufficiently large $q$} every matrix in $M_n(\F_q)$ can be written as a sum of $G(k,n)$ number of $k$th powers (of matrices in $M_n(\F_q)$). It is well-known, due to Small \cite{Sm}, that $G(k,1) = 2$ for all $k \geq 1$; see Proposition \ref{n1prop}.  Including this result of Small, we prove that, 

\begin{center}
For all $k \geq 2$, $G(k,1)= G(k,2) = 2$ and 
$G(k,n) \leq 3$ for all $n \geq 3$.
\end{center}

\non In other words, we prove the following results. 

\begin{thm}\label{intro_main_thm_1}
For all integers $k \geq 1$, there exists a constant $C_k$ depending only on $k$, such that for all $q > C_k$, every matrix in $M_2(\F_q)$ is a sum of two $k$th powers.
\end{thm}

\begin{thm}\label{intro_main_thm_2}
For all integers $k  \geq 1$, there exists a constant $C_k$ depending only on $k$, such that for all $q > C_k$, and for all $n \geq 3$, every matrix in $M_n(\F_q)$ is a sum of \textit{at most three} $k$th powers.
\end{thm}

\non We note here that our results go beyond an eariler result due to Demiroglu published in this journal \cite{De}. While Demiroglu obtained the result with infinitely many exceptions, we have the same result with only finitely many exceptions. Furthermore, the methods used to obtained these results in this article are distinct from the ones used by Demiroglu. These results are consequences of Weil's result on the number of solutions in $\F_q$ of polynomials defined over $\F_q$;  see Theorem \ref{weil_thm}. \\

\non The paper is organized as follows. In \S \ref{Weil_results} we obtain consequences of Weil's results on the number of solutions to equations over finite fields. In \S \ref{n2_section} we prove Theorem \ref{intro_main_thm_1}. In \S \ref{n3_section} we use results of \S \ref{n2_section} to prove Theorem \ref{intro_main_thm_2}. \\

\section*{Notation}
The diagonal matrix with entries $\lam_1, \ldots, \lam_n$ along the diagonal is denoted by $\diag(\lam_1, \ldots, \lam_n)$. \\

\non The direct sum $A \oplus B$ of two square matrices $A$ of size $r$ and $B$ of size $s$ is the square matrix of size $r+ s$ with $A$ and $B$ lying along the diagonal; that is $ A \oplus B$ is in the block diagonal form.\\

\non $\F_q$ always denotes the finite field with $q$ elements.

\section{Equations over Finite Fields}\label{Weil_results}

 First, we need some results from algebraic geometry over finite fields. Recall, a polynomial over $\F_q$ is absolutely irreducible if it remains irreducible over the algebraic closure of $\F_q$.

\begin{lem}\label{abs_irr_lem}
Consider the polynomial $Y^d = f(X)$ in $\F_q[X,Y]$, where $d \geq 1$. The following conditions are equivalent:
\begin{enumerate}
  \item
  $Y^d - f(X)$ is absolutely irreducible.
  \item
  $Y^d - c f(X)$ is absolutely irreducible for every 
  nonzero $c$ in $\F_q$.
  \item
  Let $f(X) = a(X - \al_1)^{d_1} \cdots (X - \al_s)^{d_s}$ be the factorization of $f$ in the algebraic closure of $\F_q$ with $\al_i \neq \al_j$ for $i \neq j$ and $d_i$ are nonnegative. Then $(d, d_1, d_2, \cdots , d_s) = 1$.

\end{enumerate}
\end{lem}
\begin{proof}
The reader may refer to [Page 11, \cite{Sc}] for a proof.
\end{proof}

\noindent For an introduction to solutions to equations over finite fields the reader may refer to \cite{IR}. The following theorem is a special case of Weil's theorem on number of solutions of equations over a finite field \cite{We}.
\begin{thm}\label{weil_thm}
(Weil) Consider the polynomial $Y^d = f(X)$ over $\F_q$, where $d \geq 1$ . Let $m$ be the degree of $f(X)$. Suppose $Y^d - f(X)$ is absolutely irreducible and  $q > 100 \cdot d \cdot m^2	$. Let $N$ be the number of zeros of the polynomial $Y^d - f(X)$. Then 
$$
\abs{N - q} \leq 4 d^{3/2} m \sqrt{q}.
$$
\end{thm}
\begin{proof}
This is a consequence of Weil's result \cite{We}. For a proof the reader may refer to [Page 10, \cite{Sc}].
\end{proof}

\begin{prop}\label{abs_irr}
Let $c$ be any nonzero element in $\F_q$, and consider the polynomial $F := X^{k} + Y^{k} - c$ in $\F_q[X,Y]$. For all sufficiently large $q$, there exists a solution $(x,y) \in \F_q \times \F_q$ such that  
\begin{enumerate}[(a)]
\item \label{sol}
$F(x, y) = 0$.
\item \label{eq}
$x^k \neq y^k$.
\item \label{prod}
Moreover, for any $\lam \in \F_p$, the solution $(x,y)$ satisfying $(1)$ and $(2)$ may be chosen so that $x^k y^k \neq \lam$.
\end{enumerate}
\end{prop}

\begin{proof} Let $p$ be the characteristic of $\F_q$. The $p$th power map on $\F_q$ is an automorphism of $\F_q$, so the equation $X^{l} + Y^l - c$ has a solution satisfying the conditions in the conclusion if and only if its $p$th power  
$$
(X^{l} + Y^{l} - d)^p = X^k + Y^k - c^p.
$$
has a solution satisfying the conditions in the conclusion. Hence we may assume without loss of generality that $p \nmid k$. The derivative of the polynomial $f(X) := - (X^k - c)$ is nonzero, so $f(X)$ is separable, and therefore by the implication $(3) \implies (1)$ of Lemma \ref{abs_irr_lem}, the polynomial $F(X,Y) := Y^k + X^k - c$ is absolutely irreducible over $\F_q$.  Therefore, by Theorem \ref{weil_thm}, the number of solutions $N$ to $F(X,Y)$ satisfies the inequality 
$$
\abs{N - q} \leq 4 k^{5/2} \sqrt{q}.
$$

\non Now the set of solutions satisfying \eqref{sol} and \eqref{eq} is finite. Indeed, if $(x,y)$ is a solution to $F$ with $x^k= y^k$ and if $\zeta_k$ is the $k$th root of unity in the algebraic closure of $\F_q$ then $(x, \zeta_k x)$ are all the possible solutions $(x,y)$ such that $x^k = y^k$. Furthermore, the set of solutions satsifying \eqref{sol} and \eqref{prod} is also finite. Indeed, let $x$ and $y$ be chosen as above satiyfing the system of equations $X^k + Y^k = c$ and $
X^k Y^k = \lam$. Then $X^k(c - X^k) = \lam$ so that $x$ is a root of the polynomial $X^{2k} - c X^k + \lam$, which has  at most $2k$ number of roots in $\F_q$. Corresponding to each root $x$, there exists at most $k$ number of $y$ such that $(x,y)$ is a solution to the system. Thus the number of solutions to the system is at most $2k^2$. Therefore, there is a solution $(x,y)$ satisfying the desired conditions \eqref{sol}, \eqref{eq} and \eqref{prod}, if $q$ is sufficiently large, for instance when $q >  400 \cdot k ^{5/2} \sqrt{q} + k + 2k^2$, so that for $C_k = (k + 2 k^2)^2$, the conclusion holds.
\end{proof}

\section{$n =1, 2$ case}\label{n2_section}

\non We begin by observing the following elementary facts. 
\begin{lem} \label{reduction_lem}
Let $k, m_1, m_2, m$ and $n$ be positive integers. Let $A, B, C, D, P$ and $Q$ be matrices over $\F_q$ for a fixed $q$.
\begin{enumerate}
  \item \label{rep}
The representation of $A$ as a sum of $k$th powers is invariant under conjugation by $GL_n(\F_q)$. In other words, if $A$ is a sum of $k$th powers, so is any conjugate of $A$.
  \item
 If $A$ and $B$ are matrices of same size $n$, and if $A$ is a sum of $m_1$ $k$th powers and $B$ is a sum of $m_2$ $k$th powers then the matrix sum $A + B$ is a sum of $(m_1 + m_2)$ $k$th powers. 
\item \label{dir_sum}
 If  $A$ is a sum of $m_1$ $k$th powers and  $B$ is a sum of $m_2$ $k$th powers, then $A \oplus B$ is a sum of $m$ $k$th powers where $m$ is the maximum of $m_1$ and $m_2$.
 
\item \label{dir_conj}
Let $A$ and $C$ be $r \times r$ matrices, and $B$ and $D$ be $s \times s$ matrices. Then  $(A \oplus B) \cdot ( C \oplus D) = AC \oplus BD$, where $\cdot$ denotes matrix multiplication. In particular, for any $k\geq 1$, $(A \oplus B)^k = A^k \oplus B^k$. Furthermore, the inverse of $A \oplus B$ exists if and only if $A$ and $B$ are invertible, in which case $(A \oplus B)^{-1} =A^{-1} \oplus B^{-1}$. In particular, if $P \in GL_r(\F_q)$ and $Q \in GL_s(\F_q)$ then 
$$
(P \oplus Q) \cdot (A  \oplus B) \cdot (P^{-1} \oplus Q^{-1}) = P A P^{-1} \oplus Q A Q^{-1}.
$$
\end{enumerate}
\end{lem}
\begin{proof}
For part \eqref{rep}, observe that if $ A = B_1^k + \ldots + B_m^k$ for some $B_i \in M_n(\F_q)$ then for $P \in GL_n(\F_q)$,
\begin{align*}
PAP^{-1} &= P (B_1^k + \ldots + B_m^k) P^{-1} = P B_1^k P^{-1} + \ldots + P B_m^k P^{-1} \\
&= (P B_1 P^{-1})^k + \ldots + (P B_m P^{-1})^k.
\end{align*}
The remaining parts are straightforward.
\end{proof}
\non The action of $GL_n(\F_q)$ on $M_n(\F_q)$ by conjugation partitions $M_n(\F_q)$
into conjugacy classes, with each conjugacy class represented by a matrix in
rational canonical form [Theorem B-3.46, \cite{Ro}]. By definition, a rational
canonical form of a matrix in $M_n(\F_q)$  is a direct sum $ C(g_1) \oplus \ldots \oplus C(g_r)$  of companion matrices $C(g_i)$, where $g_i$ are monic polynomials in $\F_q[x]$, such that $g_1 \mid g_2 \mid \ldots \mid g_{s}$. For $g := x^n -a_{n-1} x^{n-1} \ldots - a_1 x - a_0$ in
$\F_q[x]$ its companion matrix $C(g)$ is the $n \times n$ matrix is
\begin{equation} \label{companion_matrix}
C(g) := 
\begin{bmatrix} 
0 & 1 & 0 & \ldots & 0  \\ 
0 & 0 & 1 & \ldots & 0 \\ 
\vdots &\vdots & \vdots & \vdots  & \vdots \\
0 & 0 & 0 & \ldots & 1 \\
a_0 &a_1 & 0 &\cdots & a_{n-1}
\end{bmatrix}
\end{equation}
By Lemma \eqref{reduction_lem} \eqref{rep} \eqref{dir_sum}, and \eqref{dir_conj}, it follows that to express a matrix as a sum of $k$th powers it suffices to express companion matrices of polynomials in $\F_q[x]$ as a sum of $k$th powers. \\

\noindent First we consider the representibility of $1 \times 1$ matrices. Consider the following well-known result due to Small \cite{Sm}.

\begin{prop}\label{n1prop}
(Small) Fix an integer $k \geq 1$. For all sufficiently large finite fields $\F_q$ with $q > k^4$, every element of $\F_q$ is a sum of two $k$th powers.
\end{prop}

In this section, we prove that every $2 \times 2$ matrix over $\F_q$ is a sum of two $k$th powers for any $k \geq 1$  and for all sufficiently large $q$. The following lemma is elementary and well known, but for the sake of completeness we demonstrate it here.
\begin{lem}\label{rep_eig}
Suppose the characteristic polynomial of a  matrix $A$ in $M_2(\F_q)$ has a repeated eigenvalue $\lam 
\in \F_q$. Then $A$ is similar either to the diagonal matrix $\diag(\lam, \lam)$ or to $\begin{bmatrix} \lam & 1 \\  0 & \lam \end{bmatrix}$.
\end{lem}
\begin{proof}
Suppose $A$ is not similar $\diag(\lam, \lam)$. Let $B := A - \lam I$ which is nonzero and nilpotent. We claim that it is similar to the nilpotent matrix $\begin{bmatrix} 0 & 1 \\ 0 & 0 \end{bmatrix}$ so that $A$ is similar to the matrix in desired form. If $v$ is an eigenvector corresponding to the eigenvalue $\lam$ of $A$, then $Bv = 0$, so that $B$ has nontrivial kernel, viewing $B$ as a $\F_q$-linear transformation from $\F_q^2 \to \F_q^2$. Choose a nonzero $v$ such that $Bv = 0$ and define a matrix $P$ whose first column is $v$ and the second column to be any vector $w$ such that $B w = v$, which exists because the kernel of $B$ is equal to the image of $B$. Then $P^{-1} B P = \begin{bmatrix} 0 & 1 \\ 0 & 0 \end{bmatrix}$ as desired.
\end{proof}

\begin{lem}\label{n2diagcase} Suppose a matrix $A$ be similar to the diagonal matrix $\diag(\lam, \mu)$ over $\F_q$. Then, for any $k \geq 1$, and for all sufficiently large $q$, the matrix $A$ is a sum of two $k$th powers.\end{lem}
\begin{proof}
The expression of a matrix as a sum of $k$th powers is invariant under conjugation, so $A$ may be assumed to be diagonal matrix $\diag(\lam ,\mu )$. By Proposition \ref{n1prop}, for all sufficiently large $q$, the elements $\lam, \mu \in \F_q$ are sums of two $k$th powers, so that $\lam = a_1^k + a_2^k$ and $\mu = b_1^k + b_2^k$, with $a_1, a_2, b_1, b_2 \in \F_q$. Then,
$$
A = \diag( \lam, \mu) = \diag(a_1^k+ a_2^k, b_1^k + b_2^k) = \diag(a_1, b_1)^k + \diag(a_2, b_2)^k.
$$
\end{proof}

\noindent Before proving the main result, we need the following technical lemma.

\begin{prop}\label{n2unipotentcase}
For any element $\lam \in \F_q$, the matrix $A := \begin{bmatrix} \lam & 1 \\ 0 & \lam \end{bmatrix}$ is a sum of two $k$th powers.
\end{prop}
\begin{proof}
Let $a, d$ be elements in $\F_q$ such that $a +d \neq 0$ and $2 \lam \neq a + d$. By Proposition \ref{abs_irr} there exist  $r, s \in \F_q$ such that
\begin{enumerate}[(a)]
\item \label{conda}
   $r \neq s $ and $r + s = a + d$ (with $c$  takes as $a + d$ in Proposition \ref{abs_irr}).
\item  \label{condb}
  $r$ and $s$ are $k$th powers.
\item \label{condc}
 $rs \neq ad$ (with `$\lam$' taken as $ad$ in Proposition \ref{abs_irr})
\end{enumerate}   
Let $c$ be an arbitrary element in $\F_q$. Define $b$ as 
$$
b = 
\begin{cases} 
0 & \textrm{ if } c = 0 \\
(-c)^{-1}(rs - ad) & \textrm{ if } c \neq 0.
\end{cases}
$$
Then $b$ is nonzero if and only if $c$ is nonzero by condition \eqref{condc}. The characteristic polynomial of the matrix $B := \begin{bmatrix} a & b \\ c & d \end{bmatrix}$ is $X^2 - (a + d ) X + (ad - bc)$, which is equal to $X^2 - (r + s) X + r s$ due to \eqref{conda} and the definition of the element $b$ above. Therefore, $B$ is similar to the diagonal matrix $\diag(r,s)$ which is a $k$th power, due to \eqref{condb}.

Now, due to Proposition \ref{abs_irr} again, there exist distinct $t, u \in \F_q$ with $a$ and $d$ are \textit{as chosen above}, 
\begin{itemize} 
  \item
  $t \neq u$ and $t + u = (\lam - a) + (\lam -d)$.
  \item 
  $t$ and $u$ are $k$th powers.
  \item
  $tu \neq (\lam -a) (\lam -d)$.
\end{itemize}
Define
$$
b' := 
\begin{cases} 
0 & \trm{ if } c = 0 \\
c^{-1}(tu - (\lam -a)(\lam -d)) & \trm{ if } c \neq 0.
\end{cases}
$$
Then the matrix $C := \begin{bmatrix} \lam - a & b' \\ -c & \lam -d \end{bmatrix}$ is conjugate to the diagonal matrix $\diag(t,u)$. 

Clearly $B + C = \begin{bmatrix} \lam & b + b' \\ 0 & \lam \end{bmatrix}$, and if $c \neq 0$ then $b' \neq - b$ so that their sum $x := b + b' \neq 0$. The  following relation
$$
\begin{bmatrix} x^{-1} & 0 \\ 0 & x \end{bmatrix} \begin{bmatrix} \lam & x \\ 0 & \lam \end{bmatrix} \begin{bmatrix} 1 & 0 \\ 0 & x \end{bmatrix} = \begin{bmatrix} \lam & 1 \\ 0 & \lam \end{bmatrix},
$$
implies that $A$ is conjugate to the sum $B + C$ and therefore it is a sum of two $k$th powers as desired.
\end{proof}

\noindent We now prove the first main result of this article. 
\begin{proof} (of Theorem \ref{intro_main_thm_1})
Let $A \in M_2(\F_q)$, and $p(t)$ denote its characteristic polynomial. We consider three cases based on the nature of splitting $p(t)$. First, if $p(t)$ splits completely over $\F_q$ with distinct roots, then $A$ is conjugate to a diagonal matrix with eigenvalues lying along the diagonal. By Lemma \ref{n2diagcase} $A$ is a sum of two $k$th powers, for all sufficiently large $q$.
Second, suppose $p(t)$ splits completely over $\F_q$ with repeated root $\lam$. Then $A$ is similar to a diagonal matrix $\diag(\lam,\lam)$ or to the upper triangular matrix $\begin{bmatrix} \lam & 1 \\ 0 & \lam \end{bmatrix}$. In the former case the assertion follows by Lemma \ref{n2diagcase} and in the latter case by Proposition \ref{n2unipotentcase}. Finally, suppose $p(t)$ is irreducible over $\F_q$. Consider the evaluation map $\ev_A : \F_q[t] \to M_2(\F_q)$ defined by sending $t$ to $A$. It is an $\F_q$-algebra homomorphism with kernel as the ideal generated by $p(t)$ and image as the $\F_q$-algebra, say $\F_q[A]$, generated by the identity matrix $I_2$ and $A$. Therefore, $\F_q[A]$ is isomorphic to $\F_{q^2}$, and by Proposition \ref{n1prop} every element in $\F_{q^2}$, hence $A$, is a sum of two $k$th powers for all sufficiently large $q$.
\end{proof}

\section{$n \geq 3$ case}\label{n3_section}
Before we state our main results, we note one fact: the cardinality of the subgroup consisting of the $k$th powers of elements of $\F_q$ is $(q-1)/ \gcd(k,q-1) + 1$, which is at least $q/k$ so that if $q$ is sufficiently large  there exists sufficiently many elements in $\F_q$ that are not $1$ and  that are $k$th powers ; see [Proposition 7.1.2, \cite{IR}].
\begin{prop}\label{elelem2}
Let $q = p^l$ with $p$ prime, $l \geq 1$. Let $k \geq 1$ be a positive integer. Then $-1$ is a $k$th power (of an element in $\F_q$)  precisely in the following cases:
\begin{enumerate} [(a)]
\item \label{p2}
 $p = 2$.
\item \label{kodd}
 $k$ is odd.
 \item \label{kevencase}
  $k$ is even of the form $2^s t$ with $s \geq 1$ and $t$ odd and $q \equiv 1 \pmod{2^{s+1}}$. 
\end{enumerate}
In the remaining cases it is a sum of two $k$th powers.
\end{prop}

\begin{proof}

The cases \eqref{p2} and \eqref{kodd} are clear. Now consider solutions to $X^k + 1 = 0$ in $\F_q$. If $q$ is even then $x = 1$ is a solution, so suppose that $q$ is odd. Write $k = 2^a b$, with $b$ odd. Then the equation $X^k +1 = 0$ has a solution  in $\F_q$ if and only if the equation $Y^{2^a} + 1 = 0$ has a solution in $\F_q$. Indeed, if $x^k + 1 = 0$ for some $x \in \F_q$, then $(x^b)^{2^a} + 1 = 0$ so that $y = x^b$ is a solution to $Y^{2^a} + 1 = 0$. Conversely, if there is $y$ in $\F_q$ such that $y^{2^a} + 1 = 0$ so that $y^{2^a} = -1$ , then $y^k = (y^{2^a})^b  = (-1)^b = -1$ so that $y$ is a solution to $x^k + 1 = 0$ in $\F_q$. Now, the equation $Y^{2^a} = - 1$ has a solution, say $y$, in $\F_q$ if and only if $(y^{2^a})^2 = 1$ so that the order of $y$ is divisible by $q-1$. In other words, $X^k + 1 = 0$ has a solution in $\F_q$ if and only if $q \equiv 1 \pmod{2^{a+1}}$. The last part is an immediate consequence of Propostion \ref{n1prop}.
\end{proof}

\non Finally, the case $n \geq 3$ can be reduced to that of $n = 1, 2$ based on the following lemma.
\begin{lem}(Botha)
Let $B$ be an $n \times n$ diagonalizable matrix over a field $\F_q$, and suppose $a \in \F_q$ is \textit{not} an eigenvalue of $B$. Then 
$$
A = \begin{bmatrix} B & v \\ 0 & a \end{bmatrix}, \trm{ where } v \in \F_q^n,
$$
is also diagonalizable.
\end{lem}
\begin{proof}
The reader may refer to [Lemma $1.1(c)$, \cite{Bo1}].
\end{proof}

\non We use a result and its proof both due to Botha [Lemma 2.1, \cite{Bo}] and use the $n =1,2$ case to prove Theorem \ref{intro_main_thm_2}. \\

\non For $n \geq 3$ even and $x \in \F_q$, let 
$$
G_n = \bigoplus_{n /2} \begin{bmatrix} 0 &  1\\ 0 & -x \end{bmatrix} \trm{ and } H_n = [0] \oplus \left( \bigoplus_{(n-2)/2} \begin{bmatrix} x & 1 \\ 0 &0 \end{bmatrix} \right) \oplus [x],
$$
For $n \geq 3$ odd and $x \in \F_q$, let 
$$
G_n = \left( \bigoplus_{(n-1)/2} \begin{bmatrix} 0 & 1 \\ 0 & - x \end{bmatrix} \right) \oplus [0] \trm{ and } H_n = [0] \bigoplus \left( \bigoplus_{(n-1)/2} \begin{bmatrix} x & 1 \\ 0 & 0 \end{bmatrix} \right)
$$

\begin{lem}(Botha)
If $\abs{\F_q} > 2$, then for any $n \geq 3$, any matrix $A$ in $M_n(\F_q)$ can be expressed as a sum of two diagonalisable matrices.
\end{lem}
\begin{proof}
The reader may refer to [Lemma $2.1$, \cite{Bo}].
\end{proof}

\non We now prove the second main result of this article:
\begin{proof} (of Theorem \ref{intro_main_thm_2}): We use the proof of the lemma to deduce our result. The following four decompositions is due to Botha \cite{Bo}, which distinguish between $n$ even and odd, and between $a_{n-1}$ zero and nonzero.

Let $A \in M_n(\F_q)$. We show that $A$ is a sum at most three $k$th powers. As noted at the beginning of \S \ref{n2_section} we may assume that  $A$ is a companion matrix \eqref{companion_matrix} of \S \ref{n2_section} . Let $v = (a_0, \ldots, a_{n-2})^T$, where $T$ denotes transpose of a matrix.\\

\non \tbf{Case $n \geq 3$ odd and $a_{n-1} = 0$}:
\begin{equation}\label{nodda0}
A = \begin{bmatrix} G_{n-1} & 0 \\ v^T & -1 \end{bmatrix} + \begin{bmatrix} H_{n-1} & e_{n-1} \\ 0 & 1 \end{bmatrix}, 
\end{equation}

\noindent where $x \neq 0,1$. Let $B$ denote the first summand, and $C$ the second summand. If $x \neq 0,1$ then the matrix $\begin{bmatrix} 0 & 0 \\ 1 & - x \end{bmatrix}$ has distinct  eigenvalues $0$ and $-x$ and the matrix $\begin{bmatrix} x & 0 \\ 1 & 0 \end{bmatrix}$ has distinct eigenvalues $0$ and $x$, so they are diagonalizable. By Lemma \ref{reduction_lem} the matrices $B$ and $C$ are diagonalizable. Furthermore let $x$ be a $k$th power distinct from $0$ and $1$; such a $x$ exists by the observation at the beginning of this section. Now, if $k$ satisfies the conclusion of Proposition \ref{elelem2}, then both $B$ and $C$ are $k$th powers so that $A$ is a sum of two $k$th powers.  In the remaining cases, $B$ is a sum of two $k$th powers by Theorem \ref{intro_main_thm_1} and $C$ is a $k$th power so that $A$	 is a sum of three $k$th powers. \\

\non \tbf{Case $n \geq 3$ odd and $a_{n-1} \neq 0$}:
\begin{equation}
A = \begin{bmatrix} G_{n-1} & 0 \\ v^T & a_{n-1} \end{bmatrix} + \begin{bmatrix} H_{n-1} & e_{n-1} \\ 0 & 0 \end{bmatrix},
\end{equation}
\noindent with $x \neq 0 , -a_{n-1}$. We choose $x$ which is a $k$th power and not equal to either $0$ or $-a_{n-1}$; such a $x$ exists by the observation at the beginning of this section. Now, if $k$ satisfies the conclusion of Lemma \ref{elelem2} then  $A$ is a sum of two $k$th powers. In the remaining cases, it follows as above that the first summand is a sum of two $k$th powers by Theorem \ref{intro_main_thm_1} and the second summand is a $k$th power so that $A$	 is a $k$th power.\\

\non \textbf{Case $n \geq 3$ even, and $a_{n-1} = 0$}:
\begin{equation}
A = \begin{bmatrix} G_{n-1} & e_n \\ 0 & -1 \end{bmatrix} + \begin{bmatrix} H_{n-1} & 0 \\ v^T & 1 \end{bmatrix},
\end{equation}
\noindent where $x \neq 0,1$. Now,let $x\neq 0,1$ and is a $k$th power. Now, if $k$ satisfies the conclusion of Lemma \ref{elelem2}, it follows that $A$ is a sum of two $k$th powers. In the remaining cases, it follows that the first summand is a sum of two $k$th powers by Theorem \ref{intro_main_thm_1} and the second summand is a $k$th power, so that $A$ is a sum of three $k$th powers.\\

\non \tbf{Case $n \geq 3$ even and $a_{n-1} \neq 0$}:

\begin{equation}
A = \begin{bmatrix} G_{n-1} & e_n \\ 0 & b_{n-1} \end{bmatrix} + \begin{bmatrix} H_{n-1} & 0 \\ v^T & c_{n-1} \end{bmatrix}.
\end{equation}
Let $a_{n-1} = b_{n-1} + c_{n-1}$ and furthermore suppose that $b_{n-1}$ and $c_{n-1}$ are $k$th powers. Such a representation exists by Proposition \ref{abs_irr}. Given such a $b_{n-1}$ and a $c_{n-1}$ there exists $x \neq 0, c_{n-1}$ which is a $k$th power. Now, if $k$ satisfies the conclusion of Lemma \ref{elelem2}, then $A$ is a sum of two $k$th powers, otherwise it is a sum of three $k$th powers.
\end{proof}

\noindent Note that we proved a stronger statement than Theorem \ref{intro_main_thm_2}, namely if $k$ satisfies one of the conditions in Lemma \ref{elelem2} then every matrix in $M_2(\F_q)$ is a sum of \tit{two} $k$th powers, otherwise it is a sum of \tit{three} $k$th powers.

\section*{Acknowledgements}
\noindent It is a great pleasure to thank Michael Larsen for suggesting the question.

\bibliographystyle{elsarticle-num} 

\begin{thebibliography}{00}
\bibitem{Bo}
J.D. Botha: \tit{Sums of diagonalizable matrices}, Linear Algebra and its Applications, 315 (2000), 1--23.

\bibitem{Bo1}
J.D. Botha: \tit{Products of diagonalizable matrices}, Linear Algebra and its
Applications, 273 (1998), 65--82.

\bibitem{De}

Yesim Demiroglu: \tit{Waring's problem in finite rings}, Journal of Pure and Applied Algebra, Vol. 223, Issue 8, (2019), 3318--3329
\bibitem{IR}
K. Ireland, M.Rosen: \tit{A classical introduction to modern number theory}, (84). Springer Science and Business Media, 2013.

\bibitem{Ro}
J.J. Rotman: \tit{Advanced Modern Algebra}: Third Edition, Part $1$, Graduate
Studies in Mathematics, American Mathematical Society, 165 (2015).

\bibitem{Sc}
W. M. Schmidt: \tit{Equations over finite fields: an elementary approach}, Lecture Notes in Mathematics, 536. New York: Springer-Verlag, 1976.

\bibitem{Sm}
C. Small: \tit{Sums of powers in large finite fields}, Proc. Amer. Math. Soc. 65 (1977), 35--36.

\bibitem{VW}
R.Ch.Vaughan, T.D. Wooley: \tit{Waring's porblem: a survey}, Number Theory Milleninum 3 (2002), 301--340.

\bibitem{We}
Weil, Andr\'e: \tit{Numbers of solutions of equations in finite fields}, Bull. Amer. Math. Soc. 55, 5 (1949) 497--508.

\end{thebibliography}

\end{document}